\documentclass[12pt]{amsart}
\usepackage{graphicx}
\usepackage{amssymb}
\usepackage{hyperref}
\usepackage{verbatim}
\usepackage[english]{babel}
\usepackage{courier} 
\normalfont
\usepackage[T1]{fontenc}

\newtheorem{theorem}{Theorem}[section]
\newtheorem{corollary}[theorem]{Corollary}
\newtheorem{lemma}[theorem]{Lemma}
\newtheorem{prop}[theorem]{Proposition}
\theoremstyle{remark}
\newtheorem{rem}[theorem]{\bf Remark}
\theoremstyle{definition}
\newtheorem{definition}[theorem]{Definition}
\theoremstyle{definition}

\theoremstyle{remark}

\theoremstyle{remark}

\newcommand{\dbar}{\bar\partial}
\newcommand{\im}{{\rm im}}
\newcommand{\tr}{{\rm tr}}
\newcommand{\trg}{{\rm tr}_G}
\newcommand{\mbar}{M\llap{\raise 1.2ex\hbox{---}}}

\newcommand{\Hmm}[1]{\leavevmode{\marginpar{\tiny%
$\hbox to 0mm{\hspace*{-0.5mm}$\leftarrow$\hss}%
\vcenter{\vrule depth 0.1mm height 0.1mm width \the\marginparwidth}%
\hbox to 0mm{\hss$\rightarrow$\hspace*{-0.5mm}}$\\\relax\raggedright #1}}}
\newcommand{\I}{\mathfrak{Im}}
\newcommand{\R}{\mathfrak{Re}}

\begin{document}

\title{Bergman spaces of natural $G$-manifolds}
\author{Giuseppe Della Sala}
\author{Joe J. Perez}

\begin{abstract}Let $G$ be a unimodular Lie group, $X$ a compact manifold with boundary, and $M$ the total space of a principal bundle $G\to M\to X$ so that $M$ is also a strongly pseudoconvex complex manifold. In this work, we show that if there exists a point $p\in bM$ such that $T_pG$ is contained in the complex tangent space $T_p^c bM$ of $bM$ at $p$, then the Bergman space of $M$ is large. Natural examples include the gauged $G$-complexifications of Heinzner, Huckleberry, and Kutzschebauch.\end{abstract}

\maketitle
{\tiny\tableofcontents}
%

\section{Introduction}

\subsection{General setting}

Let $M$ be a complex manifold with nonempty, smooth, strongly pseudoconvex boundary $bM$ and $\bar M=M\cup bM$ so that $M$ is the interior of $\bar M$, and ${\rm dim}_{\mathbb C}(M)=n+1$.  Also assume that $\bar M$ is a closed subset in $\widetilde{M}$, a complex neighborhood of $\bar M$ of the same dimension on which the complex structure on $\widetilde{M}$ extends that of $M$ and every point of $\bar M$ is an interior point of $\widetilde{M}$. Denote by $\mathcal O(M)$ the space of holomorphic functions on $M$.

Suppose that a Lie group $G$ acts by biholomorphisms on $\widetilde M$, and that $M$ is invariant under this action. We assume that the $G$ acts on $\overline M$  freely, properly, and cocompactly. This means that the map $G\times M \to M\times M$ given by $(g,p)\to (g\cdot p, p)$ is proper, that the stabilizer of each $p\in M$ is trivial, and that the quotient space $\overline M/G$ is compact, where from now on we denote by $g\cdot p$ the action of the group element $g\in G$ on the point $p\in M$.

Choosing a $G$-invariant measure $\mu$, smooth on $\widetilde{M}$, and restricting it to $\bar M$, we define the \emph{Bergman space} $L^2\mathcal O(M,\mu)$ as $L^2(M,\mu)\cap\mathcal O(M)$. Compactness of the quotient and invariance of the measure imply that $L^2\mathcal O(M,\mu)$ does not depend on $\mu$, provided that the measure is smooth on $\overline M$. Thus we will omit mention of the measure henceforth and write $L^2\mathcal O(M)$ for the Bergman space. 

In the following result, we use the notation $T^c_p(bM)$ for the complex tangent space $T_p(bM)\cap iT_p(bM)$, and the unimodularity assumption means that the Lie group $G$ admits a biinvariant Haar measure, see \cite{P}.
\begin{theorem}\label{main}
Let $G$ be a unimodular group acting properly, freely and cocompactly by biholomorphisms on a complex manifold $M$ with strongly pseudoconvex boundary $bM$ and suppose that, for some $p\in bM$, we have $T_p(G\cdot p)\subset T^c_p(bM)$. Then we have ${\rm dim}_G (L^2\mathcal O(M))= \infty$.\end{theorem}
\begin{rem} A thorough description of the $G$-dimension can be found in \cite{P}, but for our purposes here let us just mention that this result implies that the Bergman space is infinite-dimensional over $\mathbb C$ as any space of positive $G$-dimension is infinite-dimensional over $\mathbb C$. An example in \cite{GHS} shows that the nontriviality of the Bergman space is not automatic, even in strongly pseudoconvex manifolds admitting a cocompact Lie group action by biholomorphisms.   \end{rem}
Since $G$ acts in $M$ by measure-preserving biholomorphisms, its action induces a unitary representation of $G$ in $L^2\mathcal O(M)$. By the methods of \cite{DSP}, it can be shown that the representation $\mathcal R$ obtained is nontrivial and in fact we have
\begin{corollary}$\ker\mathcal R$ is compact.\end{corollary}

Our method of analyzing the Bergman space is derived from a partial differential equations approach to several complex variables due to Morrey, Spencer, Kohn, and others. This involves the analysis of a self-adjoint boundary value problem for an operator $\square$ similar to the Hodge Laplacian and is called the $\dbar$-Neumann problem. 

For any integers $p,q$ with $0\leq p,q\leq n+1$ denote by
$C^\infty(M,\Lambda^{p,q})$ the space of all $C^\infty$ forms of type $(p,q)$ on
$M$. 

The antiholomorphic exterior derivative $\dbar = \dbar|_{p,q}$ defines a linear map
$\dbar:C^\infty(M,\Lambda^{p,q})\to C^\infty(M,\Lambda^{p,q+1})$.
If it can be established that
\[\bar\partial u=\phi\] 
has a square-integrable, smooth solution $u$ whenever $\phi\in L^2$ belongs to $C^\infty(\bar M, \Lambda^{0,1})$ and satisfies the compatibility condition $\bar\partial\phi=0$, then we may construct $L^2$ holomorphic functions.  The first step is to use the pseudoconvexity property of the boundary to construct a function $f\in L^2$, holomorphic in a neighborhood $U_x$ of $x$ in $\bar M$, that blows up just at $x$. The function $f$ is usually chosen to be the reciprocal of the {\it Levi polynomial}.

Next, we can take a smooth function $\chi$ with support in $U_x$ that is identically equal 1 close to $x$.  Extending $\chi f$ by zero on the rest of $M$, we obtain a function, which we also call $\chi f$, defined everywhere and smooth away from $x$.  Furthermore, $\bar\partial(\chi f)=(\bar\partial\chi)f = 0$ near $x$, so $\bar\partial\chi f$ can be extended smoothly to the boundary.  If we can now find a smooth solution $u$ to $\bar\partial u = \bar\partial\chi f$, then $\Phi=\chi f-u$ will be holomorphic and must be singular at $x$ since $u$ is smooth up to the boundary. In particular, $\Phi$ will be nontrivial. 

Let us describe the construction of solutions $u\in L^{2}(M)$ to $\bar\partial u=\phi$ with $\phi\in L^{2}(M,\Lambda^{0,1})$, $\bar\partial \phi =0$, noting that solutions will only be determined modulo the kernel of $\bar\partial$ (i.e. the square-integrable holomorphic functions). First, we pass to an analysis of a self-adjoint operator as follows. Since the Hilbert space adjoint $\bar \partial ^{*}$ of $\bar\partial$ satisfies $\overline{{\rm im}\bar\partial^*}=({\rm ker}\bar\partial)^{\perp}$,
we will look for $u$ of the form $u=\bar\partial^*v$ satisfying
\begin{equation}\label{boh}\bar\partial\bar\partial^*v =\phi.
\end{equation}
In order to eliminate the compatibility 
condition on $\phi$, let us add a term, $\bar\partial^{*}\bar\partial v$, to obtain a new operator and equation  
\begin{equation}\label{KL}
\square v := (\bar\partial\bar\partial^{*}+\bar\partial^{*}\bar\partial)v=\phi,
\end{equation}
\noindent
in which $\phi$ need not be assumed to satisfy $\bar\partial \phi=0$. By von Neumann's theorem, the $\dbar${\it -Neumann Laplacian} $\square$ is self-adjoint with its natural boundary conditions and, when $\bar\partial \phi =0$ is true, Eq.\ \eqref{KL} reduces to Eq.\ \eqref{boh}. This is shown by applying $\bar\partial$ to Eq.\ \eqref{KL}, which gives that $\bar\partial\bar\partial^{*}\bar\partial v =0$ and in turn implies $0=\langle\bar\partial\bar\partial^{*}\bar\partial v, \bar\partial v \rangle= \|\bar\partial^{*}\bar\partial v\|_{L^2(M)}^2$. Thus the new term in Eq.\ \eqref{KL} vanishes when the compatibility condition holds. 

The Laplacian is elliptic but its natural $\dbar$-Neumann boundary conditions are not. Still, it turns out that the gain at the boundary depends on the geometry of the boundary, and the best such situation is that which we assume, in which the boundary is {\it strongly pseudoconvex}. In this case, the operator gains one degree on the Sobolev scale in neighborhoods of $bM$ and so global estimates including both interior and boundary neighborhoods gain only one degree. 

Our method of solution of Eq.\ \eqref{KL} is to establish that $\square$ satisfies a generalized Fredholm property. Classically, this means that the spaces ${\rm ker}\square$ and ${\rm coker}\square$ are finite-dimensional, {\it i.e.}\ on a (closed) subspace of finite complex codimension in $L^2(M, \Lambda^{p,q})$, $q>0$, the Laplacian is a Hilbert space isomorphism. In the case of compact $\bar M$, we can use the Rellich lemma and the operator's Sobolev gain to obtain the finite-dimensionality of those spaces. This completely characterizes the solvability of $\square$.

\subsection{The $G$-Fredholm property of the $\dbar$-Neumann problem} In the case of a noncompact manifold, an operator $A$ may fail to be Fredholm regardless of its Sobolev gain. That is, the kernel and/or cokernel of $A$ may be infinite-dimensional and/or the image of $A$ may not be closed. Thus, our solvability theory in the $G$-manifold case will be worked in terms of a generalized Fredholm property valid for $G$-invariant subspaces of $L^2$. The corresponding dimension is obtained by replacing the ordinary meaning of the complex dimension by the value of von Neumann's $G$-trace $\tr_G$ which takes finite values on some closed, $G$-invariant subspaces of $L^2(M)$ which are infinite-dimensional over $\mathbb C$. For this dimension to be defined, we need that $M$'s symmetry group $G$ be a unimodular group, \cite{P}, and that the quotient $X=\overline M/G$ be a compact manifold.

Making appropriate choices of metric on $M$ and in the vector bundles over $M$ and using a Haar measure on $G$, we obtain Hilbert spaces of sections on which the $G$-action is unitary. This action allows us to define an trace ${\rm tr}_G$ in the algebra of operators commuting with the action of $G$. 

For the case in which $M=G$, the $G$-dimension has a simple definition. The algebra of operators $\mathcal L_G\subset\mathcal B(L^2(G))$ commuting with the right action of $G$ is a von Neumann algebra consisting of some left convolutions $\lambda_\kappa$ against distributions $\kappa$ on $G$. On this algebra there is a unique trace $\trg$
agreeing with
\[\trg(\lambda_\kappa^*\lambda_\kappa) = \int_G ds\, |\kappa(s)|^2, \]    
\noindent
whenever $\lambda_\kappa\in\mathcal B(L^2(G))$ and $\kappa\in L^2(G)$,  \cite[\S\S 5.1, 7.2]{P}. In order to measure invariant projections in $\mathcal B(L^2(M))^G$, one uses the tensor product ${\rm Tr}_G = \tr_G\otimes \tr_{L^2(M/G)}$. Restricting this trace to orthogonal projections $P_L$ onto $G$-invariant subspaces $L$ provides a dimension function $\dim_G$,
\[\dim_G(L) = {\rm Tr}_G(P_L).\]
With this idea of dimension, one generalizes the classical definition of Fredholm operator.
\begin{definition} A $G$-invariant operator $A:\mathcal H_1\to\mathcal H_2$ is said to be $G$-\emph{Fredholm} if $\dim_G\ker A<\infty$ and if there exists a closed, invariant subspace $Q\subset {\rm im}(A)$ so that $\dim_G(\mathcal H_2\ominus Q)<\infty$.\end{definition}
The principal result of \cite{P1} is the following: 
\begin{theorem}\label{citfromp1} Let $M$ be a complex manifold with boundary which is strongly pseudoconvex. Let $G$ be a unimodular Lie group acting freely and properly by holomorphic transformations on $M$ so that $M/G$ is compact. Then, for $q>0$, the $\dbar$-Neumann Laplacian $\square$ in $L^2(M,\Lambda^{p,q})$ is $G$-Fredholm.\end{theorem}
In other words, as the Laplacian $G$-Fredholm, it is an isomorphism in the orthogonal complement of a closed, invariant subspace of finite $G$-dimension.
\begin{rem}Examples of manifolds satisfying the hypotheses of the theorem are the gauged $G$-complexifications of \cite{HHK} for unimodular Lie groups. The unimodularity of $G$ is necessary for the definition of the $G$-Fredholm property, \cite{P}.\end{rem}

\subsection{The Levi problem on $G$-bundles} The $G$-Fredholm property established earlier provides that the image of the Laplacian nontrivially intersects any closed, invariant subspace of $L^2(M)$ of large enough dimension. On the other hand, the Paley-Wiener theorem of \cite{AL} provides that if a closed, invariant subspace of $L^2(M)$ contains an element with compact support, then it is infinite-$G$-dimensional. Thus, if $\langle u\rangle$ denotes the closure of the complex vector space generated by translates of $u\in C^\infty_c$, we have $\im\square\cap\langle\dbar\chi f\rangle\neq\{0\}$, which is the basis of the construction of the Bergman space. To be more precise, we need subspaces in $\im\square\cap\langle\dbar\chi f\rangle\cap C^\infty(\bar M, \Lambda^{0,1})$ of arbitrarily large $G$-dimension in order to solve our problem. These are constructed in \cite{P2} as images under $\dbar$ of subspaces of $L^2$ generated by convolutions as follows,
\[\langle\langle\chi f \rangle\rangle_\delta = \{R_\Delta \chi f: \Delta\in\im P_\delta\},\]
where $(R_\Delta u)(p) = \int_G dt\, \Delta(t)u(pt)$, $R_\Delta\dbar u = \dbar R_\Delta u$, and $P_\delta$ is some invariant projection in $L^2(G)$. If the projections $P_\delta$ are chosen appropriately, the convolution kernels $\Delta\in\im P_\delta$ will be smooth and so elements $u\in\langle\langle\chi f \rangle\rangle_\delta$ will have $\dbar u\in C^\infty(\bar M,\Lambda^{0,1})$. If it happens that $u\notin C^\infty(\bar M)$, then Kohn's nontriviality argument can be applied with $u$ replacing $\chi f$. This motivates the introduction of the property called amenability in \cite{P2}:

\begin{definition}\label{amen} Let $G\to M\stackrel p \to X$ be a principal $G$-bundle and let $\xi:\bar X\to M$ be a piecewise continuous section so that $\xi|_{p({\rm supp}\chi)}$ is continuous. The action of $G$ on $M$ is called \emph{amenable} if there exist an $x\in bM$ and $\tau>0$ so that if $f$ is a Levi polynomial at $x$, then 1) $\chi f^{-\tau}\in L^2(M)$, 2) $\|\chi f^{-\tau}(\cdot,\xi)\|_{L^1(G)}<\infty$ for all $\xi\in p({\rm supp}\chi)$, and 3) for any nonzero $\Delta\in C^\infty(G)$, we have $R_\Delta\chi f^{-\tau}\notin C^\infty(\bar M)$.\end{definition}
\begin{rem} Let us point out that our nonstandard use of the term amenable refers to actions rather than groups intrinsically and is unrelated to the existence of an invariant mean on the group as in the property due to von Neumann. 

In \cite{P2} it is shown that conditions 1) and 2) hold for $0<2\tau<\dim G$. \end{rem}
The main result of \cite{P2} and the motivation for much of the present work is
\begin{theorem} Let $G$ and $M$ be as in Thm.\ \ref{citfromp1} and assume that the action of $G$ in $M$ is amenable. It follows that the Bergman space $L^2\mathcal O(M)$ is infinite-$G$-dimensional.\end{theorem}
\begin{rem} In \cite{DSP} we apply this theorem to establish the nontriviality of the Bergman spaces of some natural $G$-manifolds; we describe these in Sect.\ \ref{HHKsds}.\end{rem}

\subsection{Integrals on the Heisenberg group}\label{Heisenbex}
Verifying in practice the condition described in Definition \ref{amen} leads, rather concretely, to the estimation of certain integrals with parameter performed on the group $G$. In this section we present an example that will serve as a model for the general situation of Theorem \ref{main}. Note that, since the integrals to estimate have a local character, a global assumptions like cocompactness of the action plays no role and in the following example it is in fact not satisfied.

Let $(z_0,z)$ be complex coordinates for $\mathbb C^{n+1}\cong \mathbb C\times \mathbb C^n$. The Siegel domain $D_{n+1}$ is defined as
\[D_{n+1} = \{(z_0,z)\in\mathbb C^{n+1}: \I z_0 > |z|^2\}\]
and the Levi polynomial of $bD_{n+1}$ at $0$ is given by $\Lambda(z_0, z) = z_0$. For $\epsilon>0$ sufficiently small, the curve $\lambda:[0,\epsilon)\to \mathbb C^{n+1}:s\mapsto (is,0)$ is a path to zero in  $D_{n+1}$. Note that this path is also normal in $\mathbb R^{2n+2}\cong\mathbb C^{n+1}$ to $T_0(bD_{n+1})$.
There is a convenient coordinatization of $bD_{n+1}$ in terms of the Heisenberg group in $n$ dimensions $H_n$ as follows. That boundary is modeled geometrically as the Lie group whose underlying manifold is $\mathbb R\times\mathbb C^n$ with coordinates $(t, \zeta_1,\zeta_2,\dots,\zeta_n)=(t, \zeta)$ and whose group law is given by
\[(t,\zeta)\cdot(t', \zeta') = (t+t'+2\,\I(\zeta\cdot\zeta'), \zeta+\zeta') \quad {\rm where} \quad \zeta\cdot \zeta'=\sum_1^n \zeta_j \bar\zeta_j'.\]
The group $H_n$ acts on $\mathbb C^{n+1}$ by holomorphic, affine transformations which preserve $D_{n+1}$ and $bD_{n+1}$ as follows: if $(t,\zeta)\in H_n$ and $z\in\mathbb C^{n+1}$, then define
\[(t, \zeta)\cdot z = \left(z_0 + t + i |\zeta|^2 + 2i\sum_1^n z_j \bar \zeta_j\ {\bf ,}\ z_1+\zeta_1,\dots, z_n+\zeta_n\right).\]
The action of a group element $(t,\zeta)\in H_n$ on the Levi polynomial $\Lambda(z)=z_0$ is then
\begin{equation}\label{action}((t, \zeta)^* \Lambda)(z) = z_0 + t + i |\zeta|^2 + 2i\sum_1^n z_j \bar \zeta_j \end{equation}
and, in particular for $(z_0,z)=(0,0)$, this expression reduces to
\begin{equation}\label{this}((t, \zeta)^* \Lambda)(0,0) = t + i |\zeta|^2. \end{equation}
Suppose that $bM$ as above coincides with $bD_{n+1}$. We will verify that the action of any subgroup $G\subset\{(0,\zeta)\}\subset H_n$ is amenable. The group $G$ consists of points $(0,\zeta)\in H_n$ with $\zeta$ belonging to a certain subset $\mathcal S\subset\mathbb C^n$ with the property $z,z'\in \mathcal S\Rightarrow z\cdot z'\in\mathbb R$, for example, $\mathcal S=\mathbb R^n\subset \mathbb C^n$.

Along the path $\lambda(s)$, we get
\begin{equation}\label{iop2}(R_\Delta \chi \Lambda^{-\tau})(\lambda(s)) =\int_{\mathcal S}d\zeta\ \frac{\Delta(0,\zeta)((0,\zeta)^* \chi(\lambda(s))}{\left[ i(s +|\zeta|^2)\right]^\tau}\end{equation}
\[\approx \int_{B}\frac{dz}{\left[ i(s +|\zeta|^2)\right]^\tau}\approx \int_0^\epsilon dr\  \frac{r^{d-1}}{\left[  i(s +r^2)\right]^\tau}.\]
Taking sufficiently many derivatives and putting $s=0$, the resulting integral is manifestly divergent, so from this we conclude that the convolution is not smooth to the boundary.

Note that we obtain a singular convolution from subgroups whose orbits satisfy the tangency condition assumed in Thm.\ \ref{main}. Our proof of the theorem depends on the fact that, as it turns out, the general case is a sufficiently small perturbation of this Heisenberg group case to preserve this divergent behavior.

\subsection{Unitary representations of Lie groups} Unitary representations of Lie groups in $L^2$-spaces of holomorphic functions have been studied intensely, and although the abstract theory of Lie group representations is highly developed, it has been long considered important to provide geometric realizations of these representations.  

The Borel-Weil theorem is an important example in which representations are realized as holomorphic functions on a space related to the group. Also, the Mackey program of construction of unitary representations of Lie groups and Harish-Chandra theory are connected to our setting, \cite{Kn}. 

As our present analytical techniques rely ultimately on the methods of the $L^2$-index theorem of Atiyah, it seems worthwhile to mention here that the first example application of that theorem in the original paper \cite{A} was in the construction of $L^2$-holomorphic representation spaces for $SL(2,\mathbb R)$ belonging to the discrete series. Though our method is a long-reaching development of this method, the initial content dates from the index theorem.

To our aesthetic, it seems most attractive to take the natural geometric, complex $G$-manifolds constructed in \cite{HHK} and investigate their Bergman spaces. Thus these manifolds will provide the starting points of our main class. 

\section{Integrals of the Levi polynomial over submanifolds}
 
 In this section we discuss the divergence of integrals analogous to (\ref{iop2}). The integrals in question will be performed over a submanifold $O\subset bM$, as well as over a $1$-parameter family of submanifolds approaching it. We think of $O$ as the orbit through $p\in bM$ of some $G$-action, but for the moment we do not make this assumption.  The validity of our approach depends on the fact that the divergence properties of these integrals are invariant under a smooth change of coordinates and of the measure, and as such, they are insensitive to the presence of a group structure on $O$. We will prove that under a certain order assumption on $O$, these divergence properties are the same as those of (\ref{iop2}). Later, we will  see that this assumption is automatically satisfied when $O$ is an orbit.
\subsection{Choice of coordinates and Levi polynomial}
As we will be interested in only the local picture, we may, without loss of generality, model our situation on a fixed, small neighborhood in $\mathbb C^{n+1}$ as follows.
\begin{definition}\label{adapt} Fix $p\in bM$. We will say that a system of local coordinates $(z_0=x_0+iy_0,z)$, where $z=(z_1=x_1+iy_1,\ldots,z_n=x_n+iy_n)$, is \emph{adapted to} $(bM,p)$ if $p\leftrightarrow (0,0)$ and $T_0(bM)$ is spanned by $\{\partial/\partial x_0, \partial/\partial x_j, \partial/\partial y_j\}_{1\leq j\leq n}$.
\end{definition}
\begin{rem}
Obviously, any set of coordinates can be brought to this form by a complex linear transformation. Moreover, if the coordinates are fixed as above, then the complex tangent space $T^c_0(bM) = T_0(bM)\cap iT_0(bM)$ is spanned by $\{\partial/\partial x_j, \partial/\partial y_j\}_{1\leq j\leq n}$.
\end{rem}
If such coordinates are fixed, we select uniquely an equation for $bM$ by solving for $\I z_0$:
\begin{equation}\label{solvingfor}
bM =\{(z_0,z)\in \mathbb C^{n+1}:\I z_0 = f(\R z_0, z)\}
\end{equation}
where $f(\R z_0,z)=O(2)$ is a smooth, real-valued function. We express the second-order Taylor expansion $f_2$ of $f$ in the following way,
\begin{equation}\label{secord} 
f_2(\R z_0, z) = \R z_0\cdot \ell(\R z_0,z) + 2\R P(z) + L(z,\overline z),
\end{equation}
where $\ell$ is a real-valued, linear function in $(\R z_0,z)$, while $P$ and $L$ correspond to the parts (involving only $z$) of type, respectively, $(2,0)$ and $(1,1)$ of the second-order expansion. The notation $L(z,\overline z)$ is used to emphasize the fact that $L$ is a real-valued polynomial of degree 2 in $z$. Since $bM$ is strongly pseudoconvex, $L(z,\overline z)$, the restriction of the Levi form to $T^c_0(bM)$, is positive and Hermitian. We define an {\it adapted Levi polynomial} $\Lambda(z_0,z)$ by
\[\Lambda(z_0,z) = z_0 - 2iP(z).\]
Note that our definition does not coincide with the standard definition of the Levi polynomial associated to the defining function $\rho(z_0,z)=\I z_0 - f(\R z_0,z)$ as in \cite{K}, for example. Nevertheless, we can show that $\Lambda$ is still a support function for $bM$ at $0$.
\begin{lemma} With $\Lambda(z_0,z)$ as defined above, for a small enough neighborhood $U$ of $0$ the following hold:
\begin{itemize}
\item $\{\Lambda = 0\}\cap U \cap bM = \{0\}$,
\item the restriction of $\Lambda$ to $M\cap U$ takes values in $\mathbb C\setminus \{w\in \mathbb C: \R w=0, \I w\leq 0 \}$.
\end{itemize}
In particular, a branch of $\log \Lambda$ is well-defined on $M\cap U$. \end{lemma}
\begin{proof} By definition, we have $\R \Lambda(z_0,z)= \R z_0 + 2\I P(z)$ and $\I \Lambda(z_0,z) = \I z_0 - 2\R P(z)$. If $q=(z_0^q,z^q)$ is such that $\Lambda(q)=0$ and $q\in bM$, we have that
\[2\R P (z^q) = \I z_0^q = 2\I P(z^q)\, \ell(2\I P(z^q),z^q) + 2\R P(z^q) + L(z^q,\overline z^q) + O(3), \]
from which we get
\[L(z^q,\overline z^q) + O(3)=0.\]
If $q$ lies in a sufficiently small neighborhood $U$ of $0$, the condition above implies that $z^q=0$, hence $z_0^q=0$ and $q=0$, which proves the first point. Let $q\in M$ be such that $\R\Lambda(q)=0$. We obtain that
\[\I z_0^q > 2\I P(z^q)\, \ell(2\I P(z^q),z^q) + 2\R P(z^q) + L(z^q,\overline z^q) + O(3),\]
which implies
\[\I \Lambda(q) = \I z_0^q - 2\R P(z^q) > L(z^q,\overline z^q) + O(3).\]
If $q$ belongs to $U$, the expression on the right-hand side is positive, thus the second point is proved.\end{proof}
\subsection{An order 3 vanishing condition}  Let $O\subset bM$ be a real submanifold of dimension $2n$ with $0\in O$.  We will show in Sect.\ \ref{gendim} how to extend our arguments to manifolds of smaller dimensions. We are now in a position to describe our main assumption.
\begin{definition} If the restriction of $\R \Lambda$ to $O$ vanishes to order at least $3$ at $0$ we will say that the \emph{order 3 condition} is satisfied.
\end{definition}
\begin{rem} Denote by $\pi$ the projection of $\mathbb C^{n+1}$ onto $T_0(bM)$, and let $\Sigma=\{\Lambda=0\}$. The order 3 condition is then equivalent to the property that $\pi(O)$ and $\pi(\Sigma)$ have order of contact at least two at zero.
\end{rem}
We now check that our hypothesis is invariantly defined. 
\begin{lemma} \label{adapt}
The order 3 condition does not depend on the choice of (adapted) local coordinates about $p$.
\end{lemma}
\begin{proof} Let $(z_0',z')$ be another set of adapted local coordinates around $p\leftrightarrow 0$. Since its differential at $0$ must preserve $T_0(bM)$, the map giving the change of coordinates between $(z_0,z)$ and $(z_0',z')$ can be expressed up to a real scaling factor as
\[z_0 = z_0' + F(z_0',z'),\ z = G(z_0',z')\]
where $F(z_0',z')=O(2)$ is a holomorphic function and $G(z_0',z'):\mathbb C^{n+1}\to \mathbb C^n$ is a holomorphic mapping whose differential with respect to the variables $z'$ has non-vanishing determinant at $0$. We denote by $F_2$ the second order expansion of $F$:
\[F_2(z_0',z') = z_0' \kappa(z_0',z') + E(z')\]
where $\kappa$ is a complex-linear function and $E$ is a homogeneous polynomial of degree $2$ in the variables $z'$. We also write $G=G_1 + O(2)$, where $G_1(z_0',z') = \alpha(z_0')+\beta(z')$ is again a complex linear map.

Let
\[\{\R z_0 = h(z), \I z_0 = f(h(z),z)\}\]
be a set of equations for the orbit $O$ around $0$ in the coordinates $(z_0,z)$. The order 3 condition gives that $h(z) = -2\I P(z) + O(3)$ and in particular, $O$ must be tangent to $T^c_0(bM)$. In the coordinates $(z_0',z')$, the hypersurface $bM$ is locally defined by
\begin{equation} \label{eqbM}
\I z_0' + \I F(z_0',z') = f(\R z_0' + \R F(z_0',z'),G(z_0',z'))
\end{equation}
and the orbit $O$ by
\[\{\R z_0'+ \R F(z_0',z') = h(G(z_0',z')), \I z_0' + \I F(z_0',z') = f(h(G(z_0',z')),G(z_0',z'))\}.\]
Let $f'(z_0',z')$ be uniquely defined in such a way that $\{\I z_0' = f'(z_0',z')\}$
is a local defining equation for $bM$. To verify the order 3 condition in the new coordinates, we compute the second-order Taylor expansion $f_2'$ of $f'$,
\[f_2'(\R z_0', z') = \R z_0' \ell'(\R z_0',z') + 2\R P'(z') + L'(z',\overline z').\]
To this end, we examine the second-order jet of (\ref{eqbM}). Obviously, we need only to consider $f_2$, $F_2$ and $G_1$:
\[\I z_0' = -\I F_2(z_0',z') + f_2(\R z_0', G_1(z_0',z')) + O(3). \]
We expand the right-hand side of the previous expression in a polynomial in $\R z_0'$, $\I z_0'$ and $z'$. Any monomial of this expansion which is of the form $Q\I z_0' $ with $Q=O(1)$  can be replaced, using (\ref{eqbM}), by $Q(-\I F(z_0',z') + f(\R z_0' + \R F(z_0',z'),G(z_0',z')))$, which is $O(3)$ and thus can be ignored. Performing the computation we obtain
\[\I z_0' = \R z_0' \ell'(\R z_0',z') - \I E(z') +2\R P(\beta(z')) + L(\beta(z'),\overline{\beta(z')}) + O(3)\]
for a suitable linear form $\ell'$, which gives in particular
\[P'(z') = iE(z')/2 + P(\beta(z')).\]
It follows that the Levi polynomial $\Lambda'(z_0',z')$ has the expression
\[\Lambda'(z_0',z')= z_0' + E(z') - 2i P(\beta(z')).\]
Now, we turn to the second-order jet of the defining equations for $O$ in the new coordinates:
\begin{align*}\R z_0' & = - \R F_2(z_0',z') + h_2(G_1(z_0',z'))\\
\I z_0' &= -\I F_2(z_0',z') + f_2(0,G_1(z_0',z')).\end{align*}
Combining these we obtain
\[\R z_0' = h'(z') =  -\R E(z') -2\I P(\beta(z')) + O(3)\]
which implies that the order 3 condition holds in the new coordinates $(z_0',z')$.\end{proof}
\begin{rem}\label{alsolower} The previous lemma also holds for a manifold $O'$ of dimension lower than $2n$. In fact, $\pi(O')$ has order of contact at least $2$ with $\pi(\Sigma)$ if and only if it is contained in a $2n$-dimensional manifold $O$ with the same property. 
\end{rem}
\subsection{Estimation of the integrals}\label{estim} By results of \cite[\S18]{FS}, about a point $p\in bM$ there exists a neighborhood $U_p\subset\widetilde{M}$ and local complex coordinates $Z=(z_0, z)\in\mathbb C\times\mathbb C^n$ such that
\begin{enumerate}
\item $p\leftrightarrow 0\in\mathbb C^{n+1}$
\item $bM =\{(z_0,z)\in\mathbb C^{n+1} : \I z_0 = |z|^2 + O(|z_0||z| + |z|^3)\}$
\end{enumerate}
\begin{rem} Condition (2) means that the hypersurface $bM$ osculates the boundary of the Siegel domain
\[D_{n+1} = \{(z_0,z)\in\mathbb C^{n+1}: \I z_0 > |z|^2\}\]
to first order at zero, and to second order along the variables $z$. This means that the surfaces $bM\cap\{z_0=0\}$ and $bD_{n+1}\cap\{z_0=0\}$ osculate to order 2. \end{rem}
In such coordinates, the Levi polynomial of $bM$ at $0$ is given by
$\Lambda(Z) = z_0$, moreover, for $\epsilon>0$ sufficiently small, the curve $\lambda:[0,\epsilon):M\to\mathbb C^{n+1}:s\mapsto (is,0)$ is a path to zero in $M$, normal in $\mathbb R^{2n+2}\cong\mathbb C^{n+1}$ to $T_0(bM)$.

We now perform an analysis on $bM$ analogous to the one in Sect.\ \ref{Heisenbex} for $bD_{n+1}\cong H_n$. The tangent space at the origin of $bM$ at 0 is given by $T_0(bM)=\{\I z_0=0\}$; denote by $\pi:\mathbb C^{n+1}\to T_0(bM)$ the orthogonal projection.  We can then express $\pi(O)$ as
\[\{(z_0, z) : \R z_0 = h(z), \I z_0=0\}\]
where $h$ is a smooth, real-valued function defined in a neighborhood of $0$ in $(0,z)$ such that $h(0)=0$ and, because of the order 3 condition, $h(z)=O(|z|^3)$. Note that, in particular, the tangent space of $\pi(O)$ at $0$ is  $\{z_0=0\}=T^c_0(bM)$.

Alternatively, we can consider a smooth parametrization $\Gamma:\mathbb C^n\to \pi(O)\subset \mathbb R\times \mathbb C^n\cong T_0(bM)$ taking the form
\[\mathbb C^n \ni \zeta \to (h(\zeta),\zeta) = (O(|\zeta|^3),\zeta).\]
Since $bM$ is given as the zero set of $\I z_0 - f$ with $f$ as in (\ref{solvingfor}), and since $O\subset bM$, it follows that the map
\[\mathbb C^n\ni \zeta\to P(\zeta):=(h(\zeta),f(h(\zeta),\zeta),\zeta) \in \mathbb C^{n+1}(\R z_0, \I z_0, z) \]
gives a smooth parametrization of $O$. 

We are interested in the following integral, keeping the notation $R_\Delta \chi \Lambda^{-\tau}$ only in analogy to (\ref{iop2}):
\[(R_\Delta \chi \Lambda^{-\tau})(0) =\int_{\mathbb C^n}d\zeta\ \frac{\Delta(P(\zeta))\chi(P(\zeta))}{\left[ \Lambda(P(\zeta))\right]^\tau}
\approx \int_{B}\frac{d\zeta}{\left[ h(\zeta)+i(f(h(\zeta),\zeta))\right]^\tau} =\]
\begin{equation}\label{thos} = \int_{B}\frac{d\zeta}{\left[ h(\zeta)+i(|\zeta|^2 + O(|\zeta|^3))\right]^\tau}, \end{equation}
where we use the notation $\approx$ to express that the quotient of the two integrands is a smooth function not vanishing at $0$. The last equality follows from the fact that $f(\R z_0,z) = |z|^2 + O(|\R z_0||z| + |z|^3)$, and that along the parametrization we have $\R z_0(P(\zeta)) = h(\zeta) = O(|\zeta|^3)$ and $z(P(\zeta)) = \zeta = O(|\zeta|)$. This integral is then a perturbation of the one computed for the Heisenberg group at $s=0$, obtained by adding the \lq\lq high order terms\rq\rq\ $h(\zeta)$ and $O(|\zeta|^3)$. Collecting $|\zeta|^2$ in the denominator of \eqref{thos},
\[\int_{B}\frac{d\zeta}{|\zeta|^{2\tau} \left(\frac{h(\zeta)}{|\zeta|^2}+ i + O(|\zeta|)\right)^\tau},\]
we have that $h(\zeta)/|\zeta|^2\to 0$ for $\zeta\to 0$. This is enough to prove that the integral diverges for large $\tau$, but, as in the case of the Heisenberg group, we actually need to look at the behavior of this integral along a one-parameter family of submanifolds rather than just along $O$.

Let us now, then, consider the path in $\mathbb C^{n+1}$ given by $\lambda(s) = (is,0)$, where $s\in \mathbb R^+$, and an arbitrary, smooth $1$-parameter family of submanifolds $O_s$ such that $\lambda(s)\in O_s$ and $O_0=O$.  We choose a map $\Gamma: \mathbb R^+ \times \mathbb C^n\to \mathbb C^{n+1}$ with the following properties:
\begin{itemize}
\item $\Gamma(s,0) = \lambda(s)$
\item for each fixed $s\in \mathbb R^+$, the map $\Gamma(s,\cdot)$ parametrizes $O_s$;
\item for each fixed $s\in \mathbb R^+$, the map $\Gamma(s,z)$ is of the form $(A(s,z)+iB(s,z),z)$.
\end{itemize}
From these properties we deduce the following expression for $\Gamma(s,z) = (A(s,z)+iB(s,z),z)$:
\begin{align*}A(s,z) &= h(z) + s(z\cdot a(s,z)), \\
B(s,z) &= f(h(z),z) + s(1+z\cdot b(s,z)) \end{align*}
for some smooth maps $a,b: \mathbb R^+ \times \mathbb C^n \to \mathbb C^n$. We remark that here we employ the notation $v\cdot w$ to denote the Euclidean scalar product of two vectors $v,w\in \mathbb C^n$.  In the following, we will replace our Levi polynomial with the more convenient $\Lambda = -iz_0$. With this new definition for $\Lambda$ we obtain the integral
\[(R_\Delta \chi \Lambda^{-\tau})(\lambda(s)) \approx \int_{B}\frac{d\zeta}{\left[ -i (h(\zeta) + s(\zeta\cdot a(s,\zeta))) + (f(h(\zeta),\zeta) + s (1 + \zeta \cdot b(s,\zeta)))\right]^\tau}, \]
which we rewrite after collecting $(s+|\zeta|^2)^\tau$ as
\[ (R_\Delta \chi \Lambda^{-\tau})(\lambda(s)) \approx   \int_{B} \frac{d\zeta}{(s+|\zeta|^2)^\tau}\frac{1}{\left[\frac{f(h(\zeta),\zeta) + s + s \zeta\cdot b(s,\zeta)}{s+|\zeta|^2} - i \frac{h(\zeta) + s\zeta \cdot a(s,\zeta)}{s+|\zeta|^2}\right]^\tau} =\]
\[= \int_{B} \frac{d\zeta}{(s+|\zeta|^2)^\tau}\frac{\cos\left(\tau\arctan\left(\theta(s,\zeta)\right)\right) + i\sin\left(\tau\arctan\left(\theta(s,\zeta)\right)\right)}{\left[\left(\frac{f(h(\zeta),\zeta) + s + s \zeta\cdot b(s,\zeta)}{s+|\zeta|^2}\right)^2 + \left(\frac{h(\zeta) + s\zeta \cdot a(s,\zeta)}{s+|\zeta|^2}\right)^2\right]^{\tau/2} } \]
where we set
\[\theta(s,\zeta) = -\frac{h(\zeta) + s\zeta \cdot a(s,\zeta)}{f(h(\zeta),\zeta) + s + s \zeta\cdot b(s,\zeta)}. \]
The key observation is the following
\begin{lemma}\label{goto0}
We have $\theta(s,\zeta)\to 0$ as $(s,\zeta)\to 0$.
\end{lemma}
\begin{proof}
We have
\[|\theta(s,\zeta)|\leq \frac{|h(\zeta)|}{|f(h(\zeta),\zeta) + s + s \zeta\cdot b(s,\zeta)|} + \frac{s |\zeta| |a(s,\zeta)|}{|f(h(\zeta),\zeta) + s + s \zeta\cdot b(s,\zeta)|} = \theta_1(s,\zeta) + \theta_2(s,\zeta).\]
Choose a small enough neighborhood $V$ of $(0,0)$ in $\mathbb R^+\times \mathbb C^n$ such that for $(s,\zeta)\in V\setminus\{(0,0)\}$ we have $s(1+\zeta\cdot b(s,\zeta))>0$. Since $f(h(\zeta),\zeta)> 0$, in such a neighborhood the denominator is positive, thus
\[\theta_1(s,\zeta) = \frac{|h(\zeta)|}{f(h(\zeta),\zeta) + s(1 + \zeta\cdot b(s,\zeta))} \leq \frac{|h(\zeta)|}{f(h(\zeta),\zeta)}= \frac{O(|\zeta|^3)}{|\zeta|^2 + O(|\zeta|^3)}\]
where the last equality is due to the order 3 condition. It follows that $\theta_1(s,\zeta)\to 0$ for $(s,\zeta)\to 0$.
For the remaining term we have
\[\theta_2(s,\zeta) \leq \frac{s |\zeta| |a(s,\zeta)|}{s(1 + \zeta\cdot b(s,\zeta))} = \frac{ O(|\zeta|)}{1 + O(|\zeta|)}  \]
which also approaches $0$ when $(s,\zeta)\to 0$.\end{proof}
For the sake of notational convenience, we also let 
\[Q(s,\zeta) = \frac{f(h(\zeta),\zeta) + s + s \zeta\cdot b(s,\zeta)}{s+|\zeta|^2},\  R(s,\zeta) = \frac{h(\zeta) + s\zeta \cdot a(s,\zeta)}{s+|\zeta|^2}.\] By arguments similar to those in Lemma \ref{goto0}, it is easy to verify that $Q(s,\zeta)\to 1$ and $R(s,\zeta)\to 0$ as $(s,\zeta)\to 0$.

Now, we consider polar coordinates $\zeta = (\Omega, r)$ in $B\setminus \{0\}\cong S\times (0,1)$ and compute our integral in these coordinates. Applying Fubini,
\[(R_\Delta \chi \Lambda^{-\tau})(\lambda(s)) \approx \int_{(0,1)} dr \frac{r^{2n-1}}{(s+|r|^2)^\tau}I(s,r)\]
where
\[I(s,r) = \int_S d\Omega \frac{\cos(\tau\arctan(\theta'(s,\Omega,r))) + i \sin(\tau\arctan(\theta'(s,\Omega,r)))}{\left(Q'^2(s,\Omega,r) + R'^2(s,\Omega,r) \right)^{\tau/2}}.\]
Here $\theta', Q'$ and $R'$ are, respectively, the expressions of $\theta, Q$ and $R$ in polar coordinates. As a consequence of Lemma \ref{goto0} and of the discussion above, we obtain that $I(s,r)\to C$ for $r\to 0$, uniformly in $s\in \mathbb R^+$, where $C=Area(S)$ is a real, positive constant. In particular, the sign of $\R I(s,r)$ is positive and bounded below for $(s,r)$ small enough. It follows that
\[\R (R_\Delta \chi \Lambda^{-\tau})(\lambda(s)) \geq C' \int_{(0,1)} dr \frac{r^{2n-1}}{(s+|r|^2)^\tau} \]
for some $C'>0$ and $s$ small enough, hence it diverges for $s\to 0$ when $\tau\gg 0$.
\subsection{General dimension}\label{gendim} Let us assume that the dimension $m$ of $O$ is lower than $2n$. By hypothesis, the tangent space of $\pi(O)$ at $0$ is contained in $T^c_0(S)=\{z_0=0\}$. Since we are principally interested in the case when $O$ is totally real, we will actually let $m\leq n$ and suppose that $T_0(\pi(O))$ is spanned by $\partial/\partial x_1,\ldots, \partial/\partial x_m$. Letting $x'=(x_1,\ldots,x_m)$, $x''=(x_{m+1},\ldots,x_n)$ and $y=(y_1,\ldots,y_n)$, a set of defining functions for $\pi(O)$ can be written in the following way:
\[\{(z_0, z) : \R z_0 = h(x'), \I z_0=0, x''=g_1(x'), y=g_2(x')\}\]
where $g_1$ and $g_2$ are vector-valued functions that vanish to first order at $0$. Moreover, because of the order 3 condition, the function $h$ vanishes up to third order at $0$. As before, we will consider a smooth parametrization $\mathbb R^m\to \pi(O)$, of the form
\[\mathbb R^m \ni \xi \to (h(\xi), \xi, g_1(\xi), g_2(\xi))\in T_0(M)\cong \mathbb R(\R z_0)\times \mathbb C^n(z)\]
where we have split the $z$-space according to the decomposition $z=(x',x'',y)$. Writing $z(\xi)$ for $(\xi,g_1(\xi),g_2(\xi))$, a parametrization of $O$ is thus given by
\[\mathbb R^m \ni \xi \to P(\xi)=(h(\xi),f(h(\xi),z(\xi)), z(\xi))\in \mathbb C^{n+1}(\R z_0, \I z_0, z).  \]
Analogously to before, we define a mapping $\Gamma: \mathbb R^+ \times \mathbb R^m\to \mathbb C^{n+1}$, with $\Gamma(s,\xi) = (A(s,\xi)+iB(s,\xi),C(s,\xi))$ where
\begin{align*}
A(s,\xi) &= h(\xi) + s(\xi\cdot a(s,\xi)), \\
B(s,\xi) &= f(h(\xi),z(\xi)) + s(1+\xi\cdot b(s,\xi)), \\
C(s,\xi) &= z(\xi) + sc(s,\xi)
\end{align*}
for some smooth maps $a,b: \mathbb R^+ \times \mathbb R^m \to \mathbb R^m$ and $c:\mathbb R^+\times \mathbb R^m\to \mathbb C^n$, parametrizing the manifold $O_s$ through the point $\lambda(s)=(is,0)$.
Again with the choice of the Levi polynomial as $\Lambda =-iz_0$, we then need to evaluate the following integral:
\begin{equation}\label{toeval} (R_\Delta \chi \Lambda^{-\tau})(\lambda(s)) \approx \int_{B}\frac{d\xi}{\left[ -i (h(\xi) + s(\xi\cdot a(s,\xi))) + (f(h(\xi),z(\xi)) + s (1 + \xi \cdot b(s,\xi)))\right]^\tau}. \end{equation}
We need only observe now that
\[ f(h(\xi),z(\xi)) = |z(\xi)|^2 + O(|\xi|^3) = |\xi|^2 + O(|\xi|^3), \]
where the first equality is due the order 3 condition and the second follows from the facts that by definition $|z(\xi)|^2 = |\xi|^2 + |g_1(\xi)|^2 + |g_2(\xi)|^2$ and $g_1(\xi),g_2(\xi)$ are both $O(|\xi|^2)$. This is all that is needed to prove the analogue of Lemma \ref{goto0} for the argument
$\theta(s,\xi) = -(h(\xi) + s\xi \cdot a(s,\xi))/(f(h(\xi),z(\xi)) + s + s \xi\cdot b(s,\xi))$. The divergence of the integral for a sufficiently large $\tau$ follows then in the same way as before, namely,
\begin{equation}\label{namely}
(R_\Delta \chi \Lambda^{-\tau})(\lambda(s)) \approx \int_{(0,1)} dr \frac{r^{m-1}}{(s+|r|^2)^\tau}I(s,r)
\end{equation}
where $I(s,r)$ approaches, uniformly in $s$, a fixed, positive constant for $r\to 0$.

\section{Some properties of the orbit}

In this section, we will establish some consequences of the hypothesis $T_p(G\cdot p)\subset T^c_p(bM)$. First of all, we derive its implications regarding the dimension of the orbit through $p$. If $S$ if a hypersurface of $\mathbb C^{n+1}$, we denote by $T(S)$ its tangent bundle, by $T^c(S)$ its complex tangent bundle and by $\mathbb C T(S), \mathbb C T^c(S)$ their respective complexifications.
\begin{lemma}
Let $S$ be a strongly pseudoconvex hypersurface of $\mathbb C^{n+1}$, $0\in S$, and let $M$ be a $CR$ submanifold of $S$, $0\in M$, such that $T(M)\subset T^c(S)$. Then $M$ is totally real and, in particular, ${\rm dim}_{\mathbb R} M\leq n$.
\end{lemma}
\begin{proof} Consider the following decompositions of the complexified tangent bundles:
\begin{align*}\mathbb C T(S) &= \mathbb C T^c(S) \oplus T = T^{1,0}(S)\oplus T^{0,1}(S)\oplus T,\\
\mathbb C T(M) &= \mathbb CT^c(M) \oplus R = T^{1,0}(M)\oplus T^{0,1}(M)\oplus R,\end{align*}
where $T$ and $R$ are the following transversal subbundles: (1) $T$ is of dimension $1$ and corresponds to the ``bad'' direction; (2) $R$ can have larger dimension. Assuming, by contradiction, that $M$ is not totally real, it follows that $T^c(M)$, and thus $T^{1,0}(M)$, are non-trivial. If $L$ is a smooth section of $T^{1,0}(M)$ in a neighborhood of $0$ in $M$, then $[L,\overline L]$ is a section of $\mathbb CT^(M)$. Hence by hypothesis, it is a section of $\mathbb CT^c(S)$ along $M$. Consider, now, any smooth extension of $L$ to a section of $T^{1,0}(S)$ over a neighborhood of $0$ in $S$, which we denote by $\widetilde L$. Note that, if $p\in M$, the value of the bracket $[\widetilde L,\overline{\widetilde L}](p)$, performed in $S$, coincides with $[L,\overline L](p)$, performed in $M$. Since $S$ is strongly pseudoconvex, we have that $\pi_T[\widetilde L, \overline{\widetilde L}](0)\neq 0$ where $\pi_T$ is the projection on the $T$-space and, in particular, $[\widetilde L, \overline{\widetilde L}](0)\not\in \mathbb CT^c_0(S)$. This is a contradiction since $[L,\overline L](0)\in \mathbb C T^c_0(S)$ as observed above.
\end{proof}
\begin{corollary} Let $G$ be a Lie group, acting freely by biholomorphisms on an $(n+1)$-dimensional complex manifold with strongly pseudoconvex boundary $S\ni 0$, and denote by $G_0=0\cdot G$ the orbit of $G$ through $0$. If $T_0(G_0)\subset T^c_0(S)$ then ${\rm dim}_{\mathbb R} G\leq n$.
\end{corollary}
\begin{proof}
Since $G$ acts by biholomorphisms, $T^c_p(G_0)$ has the same dimension at every $p\in G_0$, which means that $G_0$ is a $CR$ submanifold of $S$. By the same reason the condition $T_0(G_0)\subset T^c_0(S)$ implies $T(G_0)\subset T^c(S)$. The previous lemma then yields ${\rm dim}_{\mathbb R} G_0\leq n$, thus ${\rm dim}_{\mathbb R} G\leq n$.
\end{proof}
\begin{rem}
It follows that, in the example of the Heisenberg group (see Sect.\ \ref{Heisenbex}), there cannot be any subgroup contained in $\{(0,z)\}$ of dimension bigger than $n$.
\end{rem}
\begin{rem}
The results above also hold, with small adaptations to the proof, when we just assume that $S$ is of finite commutator type rather than strongly pseudoconvex.
\end{rem}
Next, we will show that, under the same hypothesis, the orbit $O=G\cdot 0$ must satisfy the order 3 condition.
\begin{prop}\label{starholds} If the orbit $O$ of $G$ through $0\in bM$ verifies $T_0(O)\subset T^c_0(bM)$, then it satisfies the order 3 condition.
\end{prop}
\begin{proof} We start by verifying the claim when ${\rm dim}_{\mathbb R} O =  {\rm dim}_{\mathbb R} G = 1$. In this case $T_e G$ is spanned by a single vector $v$. The image of $v$ by the differential of the action is a vector field $V$, defined in a neighborhood of $0$ in $M$, which is tangent to every orbit of $G$. Since $G$ acts by biholomorphisms, we have $V=\R Z$ for a non singular holomorphic vector field $Z$.

By a choice of complex coordinates $(z_0,z_1,\ldots,z_n)=(z_0,z)$, we can assume that $Z = \partial / \partial z_1$, so that $V = \partial / \partial x_1$ and the orbits of $G$ are parametrized by $G\cong \mathbb R \ni t\to (z_0,t,\ldots,z_n)\in \mathbb C^{n+1}$. Note that, by hypothesis, $\partial/\partial x_1 \in T^c_0(bM)$ hence, up to a linear transformation, we can assume that $T^c_0(M)$ is spanned by $\partial/\partial x_j,\partial/\partial y_j$ ($1\leq j\leq n$), {\it i.e.}\ that $(z_0,z)$ are adapted coordinates. Choose, as before, a local defining equation for $bM$ of the form $\{\I z_0 = f(\R z_0,z)\}$. Since $bM$ is $G$-invariant, $f$ does not depend on the variable $x_1$. We express the second order expansion $f_2$ of $f$ according to Eq.\ \eqref{secord} and we concentrate on $P(z)$, the homogeneous holomorphic polynomial of degree $2$ giving the harmonic part of the expansion. We can write
\[P(z) = Q(z') + z_1\ell(z') + \alpha z_1^2\]
where $Q(z')$ and $\ell(z')$ are, respectively, a homogeneous polynomial of degree $2$ and a complex linear function in $z'=(z_2,\ldots,z_n)$, and where $\alpha \in \mathbb C$. We claim that $\alpha\in \mathbb R$. Indeed, consider the expression of $f_2$ in the real coordinates $(x_j,y_j)_{0\leq j\leq n}$. Since $f_2$ does not depend on $x_1$, the only one of its monomials which includes only the variables $(x_1,y_1)$ is of the form $cy_1^2$ for some $c\in \mathbb R$. Since $cy_1^2 = -c/4(z_1^2 + \overline z_1^2-2z_1\overline z_1)$, it follows that $\alpha = -c/4\in \mathbb R$. In particular, if for $t\in \mathbb R$ we define $z(t) = (t,0,\ldots,0)$, we have $P(z(t)) = \alpha t^2$, which implies $\I P(z(t))\equiv 0$. Let $\gamma(t)=(0,t,0,\ldots,0) = (0,z(t))\in \mathbb C^{n+1}$ parametrize the orbit $O$ of $G$ through $0$ -- also notice that in these coordinates $O$ coincides with $\pi(O)$. With $\Lambda(z_0,z) = z_0 - 2i P(z)$, we have 
\[\R \Lambda(\gamma(t)) = 2 \I P(z(t))\equiv 0,\] 
thus the order 3 condition is certainly satisfied. By Lemma \ref{adapt} and Remark \ref{alsolower}, the condition is also verified in the original coordinates.
 
Let us now turn to the general case. Let $\mathfrak g = T_e G$ be the Lie algebra of $G$, and let $U$ be a neighborhood of $0$ in $\mathfrak g$ such that the exponential map is a diffeomorphism $U\to \exp(U)$ onto a neighborhood of the identity in $G$. It follows that the map $\Gamma: U\to O$ defined as
\[U\ni v\to \Gamma(v) = \exp(v)\cdot 0\in O\subset T_0(bM) \]
gives a regular parametrization of a neighborhood of $0$ in $O$. We must verify that the function $\R \Lambda(\Gamma(v))$, defined on $U$, vanishes to third order at $0$, which is equivalent to verifying that, for every fixed $w\in \mathfrak g$ with $|w|=1$, the function $\mathbb R \ni t \to \R \Lambda(\Gamma(tw))$ vanishes to third order at $0$. This is the same as checking the order 3 condition for the (local) $1$-parameter real subgroup $G_w$ of $G$ generated by $w$; since of course the tangent space of the orbit $O_w$ of $G_w$ is contained in $T_0(O)\subset T^c_0(bM)$, by the discussion above we have that our condition is satisfied for $O_w$.
\end{proof}
%
\section{The tangency condition implies amenability}

Let $\mu$ be a smooth $G$-invariant measure on $M$, and choose $p\in bM$ for which the tangency condition is satisfied. A small $G$-invariant neighborhood $\mathcal U$ of $p$ in $\widetilde M$ is diffeomorphic to $B^k\times G$ with $k=2n+2-d$. The pull-back of $\mu|_{\mathcal U}$ by a diffeomorphism $\phi: B^k\times G\to\mathcal U$ is then a product measure $\nu'\otimes \nu$, where $\nu'$ is some smooth measure on $B^k$ and the Haar measure $\nu$ is biinvariant since $G$ is unimodular.

If $\Delta\in L^1(G)$, a convolution operator $R_\Delta$ is defined on $L^2(M)$ (see \cite{P2}) by
\[(R_\Delta u)(z) = \int_G d\nu(t) \Delta(t)u(z\cdot t)\]
where $z\in M$, $t\in G$ and $u\in L^2(M)$.
\begin{lemma} Let $\Delta\in L^2(G)$ and $h\in L^1(M)$. Then $R_\Delta h\in L^2(M)$.
\end{lemma}
\begin{proof} See \cite[Lemma 7]{DSP}. \end{proof}
In particular, if $\Lambda$ is the Levi polynomial at $p$, $\chi\in C^{\infty}_c(\widetilde M)$ and $0<\tau<d/2$, we have $R_\Delta \chi \Lambda^{-\tau}\in L^2(M)$ since in this case $\chi\Lambda^{-\tau}\in L^1(M)$, \cite{P2}. We also remark that the set of $\Delta$ which are admissible in the definition of amenability given in there is a $G$-invariant, smooth subspace of $L^2(G)$,  thus, by translating and rescaling $\Delta$ we can assume $\Delta(e)=1$. Our aim is to show that $R_\Delta \chi \Lambda^{-\tau}$ does not extend smoothly through $p$, and in order to do so we will look at its behavior along a certain curve $[0,1]\ni s \to \lambda(s)\in \overline M$ ending at $p$. Note that, if $\chi$ is suitably chosen, the support of $\chi\Lambda^{-\tau}$ lies on $\mathcal U$.

Also, we can choose the diffeomorphism $\phi$ mentioned above in such a way that $\phi(0,e)=p$. Let $U$ be any neighborhood of $e$ in $G$ and let $u=\chi\Lambda^{-\tau}$. It follows that
\[(R_\Delta u)(\lambda(s)) = \int_U d\nu(t) \Delta(t)u(\lambda(s)\cdot t) + \int_{G\setminus U} d\nu(t) \Delta(t)u(\lambda(s)\cdot t) = R_1(s) + R_2(s)\]
and $R_2(s)$ is a smooth function since $u(\lambda(s)\cdot t)$ is smooth and bounded for $(s,t)\in [0,1]\times (G\setminus U)$. Thus, we shall concentrate on $R_1(s)$ for a small enough neighborhood $U$. Setting $v = u\circ \phi$ and $\gamma(s) =  \phi^{-1}(\lambda(s))$, $\gamma(s) = (\gamma_1(s),\gamma_2(s))\subset B^k\times U$, we can rewrite $R_1(s)$ as
\[R_1(s) = \int_U d\nu(t) \Delta(t) v(\gamma_1(s),\gamma_2(s)\cdot t).\]
We will show that $\R (\partial^jR_1(s)/\partial s^j )\to \infty$ as $s\to 0$ by giving a lower estimate for the real part of the integrand. We point out that this property does not depend on the choice of a smooth measure on $B^k\times U$, or, for our purposes, just of a family of smooth measures $\mu_s$ on $\{\lambda(s)\}\times U$ varying smoothly with respect to the parameter $s\in [0,1]$. After a choice of local adapted coordinates in a neighborhood of $p\in bM$, consider the path $\lambda(s)$ defined in Sect.\ \ref{estim}, and the parametrization $(s,\xi)\to \Gamma(s,\xi)$ of the orbits of $G$ through $\lambda(s)$ given in Sect.\ \ref{gendim}. The push-forward of $d\xi$ by $\Gamma$ induces, in a neighborhood of $p$, a measure on the orbits of $G$ through $\lambda$. By the discussion above, we may define on a small enough neighborhood $U$, $\gamma_1([0,1])\times U$, the pull-back of this measure by the diffeomorphism $\phi$. With this choice, the integral $R_1(s)$ becomes precisely the one considered in (\ref{toeval}).
We can now conclude the proof of Thm.\ \ref{main}.
\begin{lemma} Let $G$ act freely on a complex manifold $M$ with boundary, and suppose that the orbit $O$ of $G$ through $p\in bM$ satisfies $T_p(O)\subset H_p(bM)$. It follows that the action of $G$ is amenable.
\end{lemma}
\begin{proof} With the complex coordinates around $p$ given in Sect.\ \ref{estim}, we take $\Lambda = -iz_0$. Choose $\tau = (d-1)/2$. Using (\ref{namely}) and noticing that the integrand is non-singular for $s\neq 0$, so that the differentiation under the integral sign is justified, we get 
\[\frac{\partial}{\partial s}(R_1(s))\approx \int_{(0,1)} dr \frac{r^{d-1}}{(s+|r|^2)^{\tau+1}}I(s,r) + \int_{(0,1)} dr \frac{r^{d-1}}{(s+|r|^2)^\tau} \frac{\partial}{\partial s}I(s,r) = I_1 + I_2.\]
Since $\partial I(s,r)/\partial s$ is smooth, $I_2$ is absolutely convergent for $s\to 0$ by the choice of $\tau$. From Lemma \ref{starholds}  we have (in the notation of Sect.\ \ref{gendim}) that $h(\xi)=O(|\xi|^3)$, so that the discussion before Eq.\ \eqref{namely} applies. Therefore $\R I_1$ is divergent for $s\to 0$, and we conclude that $R_1(s)$ is not smooth at $s=0$. Hence $R_\Delta \chi \Lambda^{-\tau}$ is not smooth at $p$ along the curve $\lambda(s)$, which shows that the action is amenable.
\end{proof}
%
\section{Complexification of free $G$-actions}\label{HHKsds}
%
\subsection{HHK tubes} Let $X$ and $G$ be, respectively, a real-analytic manifold of dimension $n$ and a Lie group acting freely and properly on $X$ by real-analytic transformations. In \cite{HHK} it is shown that any such $G$-action can be extended to a free and proper action by biholomorphisms on a neighborhood of $X$ in its complexification $X^{\mathbb C} \supset X$. In the same work, the authors also construct a $G$-invariant, strongly plurisubharmonic, non-negative function $\varphi$ which vanishes on $X$, thus by setting
\begin{equation}\label{hhktube}M_\epsilon = \{\varphi <\epsilon\}\subset X^{\mathbb C}\end{equation}
\noindent
for $\epsilon>0$ sufficiently small, one obtains a strongly pseudoconvex $G$-manifold $M_\epsilon$ on which $G$ acts freely by holomorphic transformations. Note, however, that the construction in \cite{HHK} also applies to the much more general case of a proper, not necessarily free $G$-action.
In the paper, the manifolds $M_\epsilon$ are called {\it gauged $G$-complexifications of} $X$. By construction, they are Stein manifolds and so possess a rich collection of holomorphic functions $\mathcal O(M_\epsilon)$ which is invariant under the induced group action. The purpose of this section is to show that the Bergman space $L^2\mathcal O(M_\epsilon)$ is also non-trivial. In order to achieve this, we prove that the sufficient condition of Thm.\ \ref{main} is satisfied when $X$ does not coincide with the underlying manifold of $G$. We refer to \cite{DSP} for a treatment of the case $G=X$.
\begin{prop}\label{HHKtubes} Let $X,G$ and $M_\epsilon$ be as in the previous paragraph, with $\dim_{\mathbb R} G<\dim_{\mathbb R} X$. If $\epsilon$ is small enough, there exists a point $p\in bM_\epsilon$ such that the tangency condition of Thm.\ \ref{main} is satisfied at $p$.
\end{prop}
In order to prove the proposition, we fix any point $q\in X$ and we choose local complex coordinates $z=(z_1,\ldots,z_n)$, $z_j=x_j+iy_j$, for a neighborhood $U$ of $q$ in $X^{\mathbb C}$ such that $q\leftrightarrow 0$ and $X\cap U\leftrightarrow \{y_j=0, j=1,\ldots,n\}$. By \cite[Lemma 3]{DSP} we may assume that up to a complex linear transformation we have
\begin{equation}\label{sumy2}\varphi(z) = \sum_{j=1}^n y_j^2 + O(|z|^3)\end{equation}
and that $T_0(G\cdot 0)$ is spanned by $\partial/\partial x_j$, $j=1,\ldots,d$ where $d=\dim_{\mathbb R} G$.

First, however, we will limit ourselves to the case where $d=1$ and $T_0(G\cdot 0)$ is spanned by $\partial/\partial x_1$.  
For any $q\in U\setminus X$, we consider the complex tangent space $H(q)=T^c_q(bM_{\varphi(q)})$ of the level set $bM_{\varphi(q)}=\{\varphi = \varphi(q)\}$ of the function $\varphi$ through $z$, and define the set $C_G=(X\cap U)\cup \{q\in U: T_q(G\cdot q)\subset H(q)\}$. In the following lemma, $J$ is the standard complex structure, orthogonality is intended with respect to the standard Euclidean metric in the coordinates $z$, and $\nabla$ denotes the gradient associated to this metric.
\begin{lemma}\label{CG} For some small neighborhood $U'$ of $0$ in $\mathbb C^n$, the set $C_G\cap U'$ is a smooth hypersurface of $U'$.
\end{lemma}
\begin{proof} Up to a holomorphic change of coordinates, we can assume that the action of $G$ is generated by the vector field $\partial/\partial x_1$; notice that we can choose this coordinate change in such a way that its linear part at $0$ is the identity, which implies that $\varphi$ still admits the expansion (\ref{sumy2}). Moreover, in the new coordinates $\varphi$ does not depend on the variable $x_1$ and $\partial/\partial x_1\in \langle \nabla \varphi(q) \rangle^\perp$ for all $q$; hence the set $C_G$ coincides with the set $\{q\in U:\partial/\partial x_1\in \langle J\nabla \varphi(q) \rangle^\perp\}$. Again by (\ref{sumy2}),
\[\nabla \varphi(z) = 2\sum_{j=1}^n y_j \frac{\partial}{\partial y_j} + O(|z|^2),\ \ J\nabla \varphi(z) = -2\sum_{j=1}^n y_j \frac{\partial}{\partial x_j} + O(|z|^2)\]
from which we derive
\[J\nabla \varphi(z) = (-2y_1 + k(z)) \frac{\partial}{\partial x_1} + V(z)\]
where $V(z)\in\langle \partial/\partial x_1 \rangle^\perp$ and $k(z)=O(|z|^2)$. It follows that $C_G$ is given by $\{-2y_1 + k(z)=0\}$; by the implicit function theorem, $C_G$ is a smooth hypersurface of a small enough neighborhood of $0$. \end{proof}
\begin{rem}\label{CGtang} From the proof of the previous lemma also follows that the tangent hyperplane of $C_G$ at $0$ is given by $\langle \partial/\partial y_1\rangle^\perp$. In fact, the gradient of the defining function $-2y_1 + k(z)$ at $0$ is a multiple of $\partial/\partial y_1$, and the same is true in the original coordinates because the linear part of the coordinate change equals the identity.
\end{rem}
\begin{proof}[Proof of Prop.\ \ref{HHKtubes}]
Now, we turn back to the case when the dimension of $G$ is arbitrary, and we define the set $C_G$ in the same way as before. Select a collection $G_1,\ldots, G_d$ of local $1$-parameter subgroups of $G$ with the property that $T_0(G_j\cdot 0)$ is generated by $\partial/\partial x_j$ for all $j=1,\ldots,d$. For $q$ in a small neighborhood $U''$ of $0$ in $\mathbb C^n$, then, we have that $T_q(G\cdot q)$ is spanned by the union of the $T_q(G_j\cdot q)$ for $1\leq j\leq d$. It follows that
\[U''\cap C_G = U''\cap \bigcap_{j=1}^d C_{G_j}.\]
By Lemma \ref{CG}, up to shrinking $U''$, each $C_{G_j}$ is a smooth hypersurface; by Remark \ref{CGtang}, then, we derive that $T_0 (C_{G_j}) = \langle \partial/\partial y_j \rangle^\perp$, which implies that the $C_{G_j}$ intersect transversally. Since $d<n$, it follows that $C_G$ is a smooth submanifold of $U''$ of real dimension strictly bigger than $n$. In particular, $C_G$ does not coincide with $X\cap U''$, and as a consequence it must intersect $bM_\epsilon$ for $\epsilon$ small enough.  Any $p\in C_G\cap bM_\epsilon$ satisfies the claim of the proposition.
\end{proof}
\subsection{Example}
Let $X=S^1_\theta \times \mathbb R_{x_1}$, $G=\mathbb R_t$, and let $T=\mathbb C/ \mathbb R$ be the complex cylinder. The complexification $X^{\mathbb C}$ of $X$ is given by $T\times \mathbb C$, in which we consider coordinates $(z_0=\theta+iy_0,z_1)$ where $\theta\in S^1, y_0\in \mathbb R, z_1\in \mathbb C$. 
 A tube $M_\epsilon$ around $X$ can be realized as a domain of $X^{\mathbb C}$ as follows:
 \[M_\epsilon = \{(z_0,z_1): y_0^2+y_1^2<\epsilon^2 \}\subset T\times \mathbb C.\]
Define, now, for any fixed $c\in \mathbb R$ an action $\phi_c$ of $\mathbb R$ on $X$ by
\[\phi_c(t)(\theta,x_1) = (\theta+ ct, x_1+t)\in S^1\times \mathbb R \]
for all $t\in \mathbb R$. For any fixed $c$, this action extends by the same formula to an action on $X^{\mathbb C}$ by biholomorphic transformations which, moreover, preserve each $M_\epsilon$. It is also clear that both the action $\phi_c$ and its extension are free and cocompact on, respectively, $X$ and $\overline M_\epsilon$. A computation shows that, indeed, the tangency condition for the action $\phi_c$ holds along the $\phi_c$-invariant submanifold of $bM_\epsilon$ given by
\[\{(z_0,z_1)\in T\times \mathbb C: y_0^2+y_1^2=\epsilon^2, y_1=-cy_0\}\subset bM_\epsilon. \]
Thus Thm.\ \ref{main} applies, showing that $\dim_GL^2\mathcal O(M_\epsilon)=\infty$. Moreover, it shows that almost every point of $bM_\epsilon$ is a (weakly) peaking point. We remark that in this case, because of the presence of cocompact lattices, the methods of \cite{GHS} already apply. However, even under the hypothesis of unimodularity a Lie group does not, in general, admit such a lattice, see \cite{R}.

\end{document}